\documentclass[12pt,leqno]{amsart}

\usepackage{amsmath,amssymb,amscd,amsthm}
\usepackage{latexsym}
\usepackage[dvips]{graphicx,epsfig}
\usepackage{graphicx}
\usepackage{color}

\newtheorem{theorem}{Theorem}
\newtheorem{lemma}{Lemma}
\newtheorem{proposition}{Proposition}

\newtheorem{definition}{Definition}
\newtheorem{corollary}{Corollary}

\newtheorem{claim}{Claim}

\newcommand{\f}[2]{\frac{#1}{#2}}
\newcommand{\dpr}[2]{\langle #1,#2 \rangle}

\newcommand{\al}{\alpha}
\newcommand{\be}{\beta}

\newcommand{\ga}{\gamma}

\newcommand{\de}{\delta}

\newcommand{\ka}{\kappa}
\newcommand{\la}{\lambda}

\newcommand{\si}{\sigma}

\newcommand{\vp}{\varphi}
\newcommand{\Vp}{\Phi}

\newcommand{\rone}{\mathbf R^1}

\DeclareMathOperator{\sech}{sech}

\newcommand{\eps}{\epsilon}

\newcommand{\cl}{\mathcal L}

\newcommand{\ch}{\mathcal H}

\newcommand{\cx}{\mathcal X}

\newcommand{\p}{\partial}

\newcommand{\beq}{\begin{equation}}
\newcommand{\eeq}{\end{equation}}
\newcommand{\beqna}{\begin{eqnarray*}}
\newcommand{\eeqna}{\end{eqnarray*}}
\newcommand{\beqn}{\begin{equation*}}
\newcommand{\eeqn}{\end{equation*}}
\newcommand{\bp}{\begin{proof}}
\newcommand{\ep}{\end{proof}}
\newcommand{\bprop}{\begin{proposition}}
\newcommand{\eprop}{\end{proposition}}
\newcommand{\bt}{\begin{theorem}}
\newcommand{\et}{\end{theorem}}
\newcommand{\bex}{\begin{Example}}
\newcommand{\eex}{\end{Example}}
\newcommand{\bc}{\begin{corollary}}
\newcommand{\ec}{\end{corollary}}
\newcommand{\bcl}{\begin{claim}}
\newcommand{\ecl}{\end{claim}}
\newcommand{\bl}{\begin{lemma}}
\newcommand{\el}{\end{lemma}}

\begin{document}

\title
[Stability of traveling waves for the  short pulse equation]
{Spectral stability for  classical periodic waves  of the Ostrovsky and short pulse models}

\author{Sevdzhan  Hakkaev}
\author{Milena Stanislavova}
\author{Atanas Stefanov}

\address{Sevdzhan Hakkaev, Faculty of Arts and Sciences, Department of Mathematics and Computer Science, Istanbul Aydin University, Istanbul, Turkey and Faculty of Mathematics and Informatics, Shumen University, Shumen, Bulgaria }\email{shakkaev@fmi.shu-bg.net}
\address{Milena Stanislavova, Department of Mathematics, University of Kansas,
1460 Jayhawk
Boulevard,  Lawrence KS 66045--7523}
\email{stanis@ku.edu}
\address{ Atanas Stefanov,
Department of Mathematics, University of Kansas,
1460 Jayhawk
Boulevard,  Lawrence KS 66045--7523}
\email{stefanov@ku.edu}

\thanks{ Milena Stanislavova is partially supported by NSF-DMS, Applied Mathematics program, under
grant \# 1516245.
Atanas Stefanov is partially supported  from NSF-DMS, Applied Mathematics program, under grant \# 1313107.}

\date{\today}


\begin{abstract}
We consider the Ostrovsky and short pulse models   in a symmetric spatial interval, subject to periodic boundary conditions. For the Ostrovsky case, we revisit the   classical periodic traveling waves and for the short pulse model, we explicitly construct traveling waves in terms of Jacobi elliptic functions. For both examples, we show spectral stability, for all values of the parameters. This is achieved by studying  the non-standard eigenvalue problems in the form $L u=\la u'$, where $L$ is a Hill operator.

\end{abstract}
\
\maketitle

\section{Introduction}
The (generalized) Korteweg-De Vries equation
\begin{equation}
\label{d:10}
u_t+\be u_{xxx}+(f(u))_x=0,
\end{equation}
is a basic model in the theory of water waves. In fact, this is one of the most ubiquitous  models in the theory of partial differential equations, modeling the unidirectional motion of waves in shallow water.  Its Cauchy problem has been comprehensively
studied in the last 50 years. Our interest is in a related model, which takes into account the effect of a (small) rotation force acting on the fluid.
  More specifically,
 \begin{equation}
\label{10}
(u_t+\be u_{xxx}+(f(u))_x)_x= \eps u, -L \leq x\leq L.
\end{equation}
 Note that we consider  \eqref{10}  on a finite nterval, with periodic boundary conditions.
 The problem on the whole line case  certainly makes sense physically, as an approximation of situations where the motion takes place on long intervals. We will however only consider the periodic case henceforth.

 We refer to  \eqref{10} as the regularized short pulse equation (RSPE), when $\beta\neq 0$.
 In  \cite{LL2, LL1} the authors   have constructed  traveling wave solutions of \eqref{10} on the whole line by employing variational methods. They have also studied the stability of such solutions by following the  Grillakis-Shatah-Strauss  arguments. Further results on the stability of these traveling waves  were obtained in  \cite{Liu1,LO}. In \cite{CMJ}, the authors have  constructed pulse solutions of \eqref{10}, for small values of $\eps$,  via singular perturbation theory.  In \cite{CMJS},  they have  shown the existence of multi-pulse solutions.  The   stability of these  waves remains an interesting open problem.

An interesting special model occurs in the absence of a KdV regularization -  in other words, $\be=0$.  This is referred to in the literature, depending on the form of the non-linearity $f$,  as the reduced Ostrovsky  or Ostrovsky-Hunter or short pulse model\footnote{Usually the models with quadratic non-linearities are referred to as Ostrovsky models, while cubic ones are referred to as short pulse models. Unfortunately,  there does not appear to be an  uniformity in this matter.}. Namely, after scaling all parameters to one, we have
\begin{equation}
\label{s1}
(u_t +(f(u))_{x})_x=u.
\end{equation}
The model \eqref{s1}, with various form of the nonlinearity has rich history, most of it unrelated to the its connections to KdV. Ostrovsky, \cite{Ostrov} in the late 70's has introduced the first model of this sort.  In the early 90's, Vakhnenko, \cite{Vakh}  proposed an alternative   derivation, while Hunter, \cite{Hunter} proposed   some numerical simulations.  The well-posedness questions were investigated by Boyd, \cite{Boyd};  Schaefer and Wayne, \cite{SW}; Stefanov-Shen-Kevrekidis, \cite{SSK}.  Liu, Pelinovsky and Sakovich, \cite{LPS1, LPS2} have studied   wave breaking, which was later supplemented by the global regularity results for small, in appropriate sense data,  of Grimshaw-Pelinovsky, \cite{GP}.  There are numerous works on explicit traveling wave solutions of these models,  \cite{GHJ, mpo,  ss2,  Stepan, ost10, ost7}.  One should note that some of this solutions are not classical solutions, but rather a multi-valued ones, \cite{ss2}.  Several authors have also explored the integrability of the Ostrovsky equation, \cite{ss2, ost7}. In particular,  they have managed to construct the traveling waves by means of the inverse scattering transform. In this regard, it is worth mentioning the very recent work \cite{JP}, where the authors study small periodic waves of the quadratic and cubic models in the form \eqref{s1}. They show orbital stability of such waves (with respect to all subharmonic perturbations!) by adapting the methods of \cite{GalP} for periodic waves of the defocussing cubic NLS. Their proof makes sense of a representation of these waves as unconstrained minimizers of appropriate functionals. Another recent development in the area is our recent paper, \cite{HSS}, which gives an explicit construction of peakon type solutions and establishes their stability.

Our main interest in this paper  is the stability of explicit traveling waves for the short pulse equation \eqref{s1}. More precisely, we follow the recent work of  \cite{GHJ}, who construct the solutions of \eqref{s1}, for $f(u)=u^2$, in terms of Jacobi elliptic functions, after a (solution dependent) change of variables.
We consider  these  solutions and we show their spectral stability with respect to co-periodic perturbations ( i.e. with respect of perturbations of the same period). In addition, we construct a family of explicit solutions in the cubic case as well. Their spectral stability for co-periodic perturbations is established as well.  In all our considerations, we consider the linearized problems after the  change of variables, where we get eigenvalue problems in the form
\begin{equation}
\label{z:15}
\cl u=\mu u'.
\end{equation}
where $\cl$ is a second order Hill operator, subject to periodic boundary conditions. Clearly, \eqref{z:15} is a non-standard eigenvalue problem, for which we develop appropriate methods to study its stability.

 We now continue on to derive the profile equations and the linearized equations.

\subsection{Profile equation and the linearized problem}
As we have alluded above, we consider \eqref{s1} with  quadratic and cubic non-linearity. Even though, one can mostly proceed to derive the profile equation
with the general  non-linearity $f(u)$, we prefer to use the explicit form in the two cases, since the specific, solution dependent transformation (see \cite{GHJ}), depends in a significant way on the particular form of $f$.
\subsubsection{The quadratic model}
In  order to derive the profile equations, we  follow the approach of \cite{GHJ}. Our first  consideration is
 the quadratic model $f(u)=\f{u^2}{2}$, the so-called Ostrovsky equation. It reads
\begin{equation}
\label{20}
(u_t+u u_x)_x = u, \ \ -L\leq x\leq L.
\end{equation}
Using the traveling wave ansatz, $u(t,x)=\vp(x - ct)$, for an unknown  periodic function $\vp$,  we arrive at the ODE,
\begin{equation}
\label{30}
((\vp - c)\vp_\xi)_\xi=\vp \ \ -L\leq \xi \leq L.
\end{equation}
 Clearly, \eqref{30}, being a  fully nonlinear equation,  is not a very nice object to deal with.  Thus, we perform a  (solution dependent) change of variables, namely
 \begin{equation}
 \label{80}
 \xi=\Xi(\eta):=\eta - \f{\Psi(\eta)}{c}, \ \ \vp(\xi)=\Phi(\eta)=\Psi'(\eta).
 \end{equation}
 If $\vp$ is an even function, so is $\Phi$ and then naturally $\Psi$ is an odd function. Compute
\begin{equation}
\label{z:25}
\f{d\Xi}{d\eta}=1-\f{\Psi'(\eta)}{c}=1-\f{\vp(\xi)}{c},
\end{equation}
so that
$$
\vp_\xi=\f{\Phi_\eta}{\f{d\Xi}{d\eta}}=  \f{c \Phi_\eta}{c-\vp(\xi)}.
$$
Thus, $(\vp-c)\vp_\xi=-c\Phi_\eta$. Taking another derivative in $\xi$,
$$
\Phi(\eta)=\vp(\xi)=((\vp - c)\vp_\xi)_\xi=- \f{c^2 \Phi_{\eta\eta}}{c-\vp(\xi)}=
- \f{c^2 \Phi_{\eta\eta}}{c-\Phi(\eta)}.
$$
We are thus lead to the profile equation
\begin{equation}
\label{per1}
c^2\Phi''=\Phi(\Phi-c).
\end{equation}
Clearly, \eqref{per1} is a standard Schr\"odinger equation, which is much easier to study. We do so in Section \ref{sec:3} below, where an explicit\footnote{in terms of Jacobi elliptic functions} expression for $\Phi$ is found. One has to keep in mind however, that the solutions of \eqref{per1} are equivalent\footnote{in appropriate sense, to be made throughout the article, in appropriate places}, so long as the transformation \eqref{80} is invertible. This is clearly requiring that the function $\eta\to \Xi(\eta)$ is monotone or equivalently, from \eqref{z:25}, that either $\vp(\xi)>c$ for each  $\xi \in [-L,L]$ or $\vp(\xi)<c$ for each  $\xi \in [-L,L]$. If that is the case, we have an interval
$[-M,M]$, so that $\Xi:[-M,M]\to [-L,L]$ is a diffeomorphism and the profile equation \eqref{per1} has to be considered with periodic boundary conditions on $[-M,M]$.

Our next task is to derive  the linearized problem for such solutions $\vp$  -
assuming that they exist and the transformation \eqref{80} is invertible in the appropriate interval.
To this end,
 we take the ansatz $u(t,x)=\vp(x - c t)+v(t, x  - c t)$ in \eqref{20} and ignore all quadratic terms. We obtain the following linearized equation
 \begin{equation}
 \label{50}
 (v_t +((\vp - c)v)_\xi)_\xi=v \ \ -L<\xi<L.
 \end{equation}
 Next, we turn \eqref{50} into an eigenvalue problem, by letting
 $v(t, \xi)=e^{\la t} w(\xi), w\in H^2[-L,L]$. This results in
 \begin{equation}
 \label{60}
 (\la w +((\vp - c )w)_\xi)_\xi=w \ \ -L<\xi<L.
 \end{equation}
 Note that \eqref{60} guarantees that $w$ is an exact derivative, which justifies our next
 change of variables $w=z_\xi$. Here, we can assume that $z \in H^3(-L,L): \int_{-L}^L z(x) dx=0$.  This can be of course always be achieved and in fact, it fixes the function $z$. Thus, we reduce matters to
  $$
  (\la z_\xi +((\vp - c)z_\xi)_\xi)_\xi=z_\xi \ \ -L<\xi<L.
  $$
 An integration in $\xi$ (and taking into account that $\int_{-L}^L z(x) dx=0$) allows us to transform the last equation into the equivalent one
  \begin{equation}
 \label{70}
 \la z_\xi +((\vp - c )z_\xi)_\xi=z \ \ -L<\xi<L.
  \end{equation}
  Indeed, in the last equation, the  constant of integration is zero, since we have an exact derivative on the left-hand side and a function of mean value zero on the right-hand side.

Now,  assume that there is an interval $[-M,M]$, so that $\Xi:[-M,M]\to [-L,L]$ is a diffeomorphism.
This appeared  previously as a necessary condition for the wave $\vp$ to exists. Denote the inverse function of $\Xi$ by $\eta(\xi)$, that is $\xi=\Xi(\eta(\xi))$.
 Introduce $Z\in L^2(-M,M)$, so that $z(\xi)=Z(\eta(\xi))$.   We have
\begin{eqnarray*}
z_\xi(\xi)=\f{Z_\eta}{\f{d\xi}{d\eta}}=\f{Z_\eta}{1-\f{\vp(\xi)}{c}}=\f{c Z_\eta}{c-\vp(\xi)}.
\end{eqnarray*}
Thus, $(c-\vp(\xi))z_\xi=c Z_\eta$ and hence
$$
[(c-\vp)z_\xi]_\xi= c \f{d}{d\xi} ( Z_\eta(\eta(\xi)))=  c \f{Z_{\eta\eta}}{\f{d\xi}{d\eta}}=
\f{c^2 Z_{\eta\eta}}{c-\vp(\xi)}=\f{c^2 Z_{\eta\eta}}{c-\Phi(\eta)}
$$
Plugging the result in \eqref{70}, we obtain
$$
\f{c^2 Z_{\eta\eta}}{c- \Phi(\eta)}=[(c-\vp)z_\xi]_\xi=\la z_\xi - z =\la  \f{c Z_\eta}{c- \Phi} - Z.
$$
All in all, we obtain the eigenvalue problem,
\begin{equation}
\label{205}
-c^2 Z_{\eta\eta} - c Z +  \Phi Z =  - \la c Z_\eta, Z\in L^2(-M,M).
\end{equation}

\subsubsection{The cubic model}
For the cubic model, we follow an identical approach, with just a slight changes to reflect the cubic nonlinearity. More precisely, let $f(u)=\f{u^3}{3}$. The profile
equation for the traveling wave solution $\vp(x-ct)$ is
\begin{equation}
\label{z:30}
((\vp^2-c)\vp_\xi)_\xi=\vp, \ \ -L<\xi<L.
\end{equation}
Next, the change of variables is of course in the form
\begin{equation}
\label{z:50}
\Xi(\eta)=\eta-\f{\Psi(\eta)}{c}, \vp(\xi)=\Phi(\eta), \Psi'(\eta)=\Phi^2(\eta)=\vp^2(\xi).
\end{equation}
Again, if $\vp$ and $\Phi$ are odd functions, then so is $\Psi$. Similar to the quadratic case,
we have
$
(\vp^2-c)\vp_\xi=-c\Phi_\eta,
$
whence
$$
\Phi(\eta)=\vp(\xi)=((\vp^2-c)\vp_\xi)_\xi=-\f{c^2 \Phi_{\eta\eta}}{c-\vp^2(\xi)}= -\f{c^2 \Phi_{\eta\eta}}{c-\Phi^2(\eta)}.
$$
Thus, we have the profile equation in the form
\begin{equation}
\label{per11}
-c^2 \Phi_{\eta\eta}-c\Phi+\Phi^3=0.
\end{equation}
Assuming that there is an interval $[-M,M]$, so that $\Xi:[-M,M]\to [-L,L]$ is a diffeomorphism, we can consider the profile equation \eqref{per11} with periodic boundary conditions on $[-M,M]$.

We now discuss the linearization around the wave $\vp(x-ct)$ for the model $(u_t+u^2 u_x)_x=u$.
Following the same steps as in the quadratic case, with $\Xi$ defined as in \eqref{z:50},  we arrive at the following linearized problem
\begin{equation}
\label{z:60}
-c^2 Z_{\eta\eta}-c Z+\Phi^2 Z=-\la c Z_\eta, Z\in L^2(-M,M).
\end{equation}
\subsubsection{Definition of spectral stability and plan of the paper}
Now that we have introduced the profile equations and the linearized problems, it is time to formally introduce the definition of stability.
For  instability we require  that \eqref{70} has a non-trivial solution $Z$ for some  $\la: \Re \la>0$. One can easily see that if $\la$ (and some $Z$) is a solution of \eqref{205} or \eqref{z:60}, then
$(-\la, Z(-\cdot))$ is also a solution. That is, there is the spectral
invariance $\la\to -\la$. Thus, instability means that there is a solution of \eqref{205} (or \eqref{z:60}) with right hand-side $\mu=-\la c>0$. If such a solution does not exist, we say that we have stability. Formally,
\begin{definition}
\label{defi:1}
Assume that the periodic wave $\vp$ is a solution of \eqref{30}, with some $c\neq 0$. Assume also that there exists a one-to-one mapping $\Xi:(-M,M)\to (-L,L), M\in (0, \infty]$ satisfying \eqref{80}.
We say the the wave is spectrally unstable, if there exists $\mu: \Re \mu>0$ and a function
$Z\in H^2[-M,M]\cap C^2(-M,M)$, so that
\begin{equation}
\label{z:62}
L[Z]:=-c^2 Z_{\eta\eta}-c Z+\Phi Z=\mu Z'.
\end{equation}
Similarly, the solution $\vp$ of \eqref{z:30} is unstable, if there is $\mu: \Re\mu>0$ and
$Z\in H^2[-M,M]\cap C^2(-M,M)$, so that
\begin{equation}
\label{z:65}
L[Z]:=-c^2 Z_{\eta\eta}-c Z+\Phi^2 Z=\mu Z'.
\end{equation}
\end{definition}

The paper is organized as follows. In Section \ref{sec:3} we first revisit the construction of the even traveling waves for the Ostrovsky model and the odd solutions for the short pulse equation. Toward the end of Section \ref{sec:3}, appropriate spectral information for the corresponding Hill operators is supplied as well.  In Section \ref{sec:33}, we develop, for the purposes of the subsequent sections,  sufficient conditions  for the positivity of a given self-adjoint operator (with finitely many negative eigenvalues) on a subspace of finite co-dimension.   In Section \ref{sec:4}, we consider the spectral stability of the waves in the quadratic (Ostrovsky) case. In Section \ref{sec:10}, we discuss the spectral stability in the cubic (short pulse) case. Finally, in Section \ref{sec:6}, we discuss the parabolic peakons for the Ostrovsky model, which can be seen as a limiting case of the
waves constructed previously. We show, that the corresponding eigenvalue problem has smooth solutions inside the interval of consideration (which however do not satisfy any periodic boundary conditions).

\section{Construction of the periodic waves and the spectral properties of the Hill operators }
 \label{sec:3}
We first discuss the construction of periodic solutions in the case of quadratic nonlinearities.

\subsection{Quadratic nonlinearities}
 Integrating once the equation (\ref{per1}), we get
    \begin{equation}\label{per2}
      \Vp_{\eta}^2=\frac{2}{3c^2}\left[ \Vp^3-\f{3c}{2}\Vp^2+A\right]=F(\Vp) ,
    \end{equation}
  where $A$ is a constant of integration.
 For  { \bf $c<0$}, in the phase plane $(\Vp, \Vp')$  equation (\ref{per1}) has
equilibra at $(0,0)$ which is saddle point and at $\left( c , 0\right)$ which is a center.
For { \bf $c>0$}, in the phase plane $(\Vp, \Vp')$  equation (\ref{per1}) has
equilibra at $(c,0)$ which is saddle point and at $\left( 0 , 0\right)$ which is a center.

Let $\Vp_0<\Vp_1<\Vp_2$ are roots of polynomial $F(\Vp)$.  We have
\begin{equation}\label{per3}
 \left|
  \begin{array}{ll}
    \Vp_0+\Vp_1+\Vp_2=\f{3c}{2}\\
   \\
   \Vp_0\Vp_1+\Vp_0\Vp_2+\Vp_1\Vp_2=0.
   \end{array}
   \right.
   \end{equation}
   Introducing a new
variable $s\in(0,1)$ via $\Vp=\Vp_0+(\Vp_1-\Vp_0)s^2$, we transform
(\ref{per2}) into $$s'^2=\alpha^2(1-s^2)(1-k^2s^2)$$ where
$\alpha$ and $k$ are positive constants  given by
\begin{equation}\label{per4}
\alpha^2=\frac{\f{3c}{2}-\Vp_1-2\Vp_0}{6c^2}, \quad
k^2=\frac{\Vp_1-\Vp_0}{\f{3c}{2}-\Vp_1-2\Vp_0}.
\end{equation}
 Therefore
\begin{equation}\label{2.5}
\Vp=\Vp_0+(\Vp_1-\Vp_0)sn^2(\alpha x;\ka).
\end{equation}
From (\ref{per3}) and (\ref{per4}, we have
  \begin{equation}\label{per5}
    \left|
      \begin{array}{ll}
        \Vp_1-\Vp_0=6c^2\alpha^2\kappa^2 \\
        \\
        \Vp_0=\f{c}{2}-2c^2\alpha^2(1+\kappa^2) \\
        \\
        \Vp_1=\f{c}{2}+2c^2\alpha^2(2\kappa^2-1)\\
        \\
        16c^2\alpha^4(1-\kappa^2+\kappa^4)=1.
      \end{array}
    \right.
   \end{equation}
For $c>0$, from (\ref{2.5}) and (\ref{per5}), and using that $sn^2(x)+cn^2(x)=1$,  we have
  $$1-\f{\varphi(\xi)}{c}=2\alpha^2c(1-2\kappa^2+\sqrt{1-\kappa^2+\kappa^4}+3cn^2(\alpha \xi; \kappa))>0$$
With this we verified that the function defined in (\ref{80}) is one-to-one.
\subsection{Cubic nonlinearities}
Here we need  periodic solutions of \eqref{per11}.
Integrating once the equation (\ref{per11}), we get
\begin{equation}\label{per12}
      \Vp_{\eta}^2=\frac{1}{2c^2}\left[ \Vp^4-2c\Vp^2+A\right]=F(\Vp) ,
    \end{equation}
  where $A$ is a constant of integration. Suppose that the polynomial $\rho^4-2c\rho^2+A$ has two positive roots $\Vp_1>\Vp_2>0$. Then, the equation (\ref{per12}) can be written in the form
  \begin{equation}\label{per13}
    \Vp_{\eta}^2=\f{1}{2c^2}(\Vp^2-\Vp_1^2)(\Vp-\Vp_2^2).
  \end{equation}
  Then the solution of the equation (\ref{per13}) is given by
    \begin{equation}\label{per14}
      \Vp(x)=\Vp_2sn(\alpha x, \kappa),
    \end{equation}
    where $-\Vp_2<\Vp(x)<\Vp_2$ and
    \begin{equation}\label{per15}
      \left| \begin{array}{ll}
        \Vp_1^2+\Vp_2^2=2c\\
        \\
        \Vp_1^2\Vp_2^2=A\\
        \\
        \kappa^2=\f{\Vp_2^2}{\Vp_1^2}, \; \; \alpha=\f{\Vp_1}{\sqrt{2}c}.
        \end{array} \right.
      \end{equation}
Note that from (\ref{per15}), $2c=\Phi_1^2+\Phi_2^2>2\Phi_2^2$, whence we have
\begin{equation}
\label{z:70}
\Vp^2-c<\Vp^2_2-c^2<0
\end{equation}

 \subsection{Quadratic nonlinearities: spectral properties of the Hill operator}

For the operator
  $$L=-c^2\partial_x^2-c+\Vp , $$
  we have the representation
  $ c=\f{1}{4\alpha^2\sqrt{1-\kappa^2+\kappa^4}}>0$ and
      \begin{equation}
    \label{400}
    L=c^2\alpha^2(-\partial_y^2+6 k^2 sn^2(y, k)-2(1+k^2+\sqrt{1-k^2+k^4})).
    \end{equation}
     It is well-known 
\cite{HIK1} that the first three eigenvalues of
     $\Lambda_1=-\partial_{y}^{2}+6k^{2}sn^{2}(y, k)$,
     with periodic boundary conditions on $[-K(k), K(k)]$ are
     simple. These eigenvalues and the corresponding eigenfunctions are:
      \begin{eqnarray}
\label{c:10}
         \nu_{0} &=& 2+2k^2-2\sqrt{1-k^2+k^4},
         \ \  \phi_{0}(y)=1-(1+k^2-\sqrt{1-k^{2}
         +k^{4}})sn^{2}(y, k),\\
\label{c:20}
         \nu_{1} &=& 4+k^{2}, \ \ \phi_{1}(y)=sn(y, k)cn(y, k)
         =-k^{-2}dn'(y, k),\\
\label{c:30}
         \nu_{2} &=& 2+2k^{2}+2\sqrt{1-k^{2}+k^{4}},
         \ \  \phi_{2}(y)=1-(1+k^{2}+\sqrt{1-k^{2}
         +k^{4}})sn^{2}(y, k).
        \end{eqnarray}

  Since the eigenvalues of $L$ and $\Lambda_1$ are related by
  $\lambda_n=\alpha^2 c^2 (\nu_n-2(1+k^2+\sqrt{1-k^2+k^4})$ in the case $c>0$ and $\lambda_n=\alpha^2 c^2 (\nu_n-2(1+k^2-\sqrt{1-k^2+k^4})$, it follows that the first three eigenvalues of the operator       $L$, equipped with periodic boundary condition on $[-K(k), K(k)]$
      are simple and $\lambda_0<0, \lambda_1<0, \lambda_2=0$ for $c>0$.

\subsection{Cubic nonlinearities: spectral properties of the Hill operator}

Here, we are interested of the spectral properties of the operator
 $$L=-c^2\partial_x^2-c+\Vp^2.$$
  From (\ref{per15}), we have
  \begin{equation}\label{cc1}
    \Vp_1^2=\f{2c}{1+\kappa^2}, \; \; \Vp_2^2=\f{2c\kappa^2}{1+\kappa^2}, \; \; \alpha^2=\f{1}{(1+\kappa^2)c}.
  \end{equation}
    Using (\ref{cc1}), we get the following representation
      \begin{equation}\label{cc2}
        L=c^2\alpha^2[-\partial_y^2+2\kappa^2sn^2(y, \kappa)-(1+\kappa^2)].
      \end{equation}
      The spectrum of
 $\Lambda_2=-\partial_y^2+2k^{2}sn^{2}(y, k)$ is formed
 by bands $[k^{2}, 1]\cup [1+k^{2}, +\infty)$. The
 first two eigenvalues and the corresponding eigenfunctions with
 periodic boundary conditions on $[-2 K(\kappa), 2 K(\kappa)]$ are simple and
 $$\begin{array}{ll}
          \epsilon_0=k^2, & \theta_0(y)=dn(y, k),\\[1mm]
          \epsilon_1=1, & \theta_1(y)=cn(y, k),\\[1mm]
          \epsilon_2=1+k^2, & \theta_2(y)=sn(y, k).
        \end{array}
      $$
       Since the eigenvalues of $L$ and $\Lambda_2$ are related by
  $\lambda_n=\alpha^2 c^2 (\epsilon_n-(1+k^2))$ , it follows that the first three  eigenvalues of the operator
      $L$, equipped with periodic boundary condition on $[-2 K(\kappa),2 K(\kappa)]$
      are simple and $\lambda_0<\la_1<0=\la_2$.

   \section{Sufficient condition for the positivity of a  self-adjoint operator positive on a finite co-dimension subspace}
   \label{sec:33}
   In this section, we develop an abstract result for positivity of self-adjoint operators, when acting on a finite co-dimension subspace of  a Hilbert space. In the applications, we would be interested in showing that a given Hill operator is positive, when restricted to a subspace, with finite co-dimension. The question then is the following - how can one characterize these subspaces or at least develop sufficient conditions for the positivity?

   In a simple situation, we have the following setup. Assume that a self-adjoint operator $\ch$, acting on a Hilbert space $\cx$ has one simple negative eigenvalue, with an eigenvector, say $\eta_0$. Clearly, $\ch|_{\eta_0^\perp}\geq 0$. This of course does not preclude the possibility that for some other vector $\xi_0$, we still have $\ch|_{\xi_0^\perp}\geq 0$.    It is reasonable to ask for some characterization (or at least sufficient condition) of such vectors $\xi_0$.

   More generally, one may ask the same question for subspaces with arbitrary finite co-dimension. Suppose that $\ch$ has $k, k\geq 1$ negative eigenvalues, counted with multiplicities, with eigenvectors say $\{\eta_1, \ldots, \eta_k\}$, which form an orthonormal system.  Denoting $Z_0=span \{\eta_1, \ldots, \eta_k\}$, we have $\ch|_{Z_0^\perp}\geq 0$. The question is again to come up with a description or at least criteria to decide which subspaces $Z$ have the property $\ch|_{Z^\perp}\geq 0$. A moment thought reveals that such subspace $Z$ must necessarily have dimension at least $k$, that is $dim(Z)\geq k=dim(Z_0)$.

 \subsection{ Positivity on a co-dimension one subspace}
 \label{sec:3.1}
  Our next result gives a sufficient condition for $\xi_0$, so that $\cl|_{\xi_0^\perp}\geq 0$.  We will apply this result to establish the stability of the waves constructed Section \ref{sec:3}, but
    the lemma  is   of  independent interest\footnote{To the best of our knowledge, this is a new resultIt is possible that we are simply unaware of
 its existence in the literature.}.
\begin{lemma}
\label{lina}
Let $(\ch, D(\ch))$ be a self-adjoint operator  on a Hilbert space $X$. Assume that
\begin{itemize}
\item $\ch$ has exactly one negative eigenvalue counted with multiplicities. That is for some $\si>0, \eta_0\neq 0, \eta_0\in D(\ch)$,
$$
\ch \eta_0 = -\si^2 \eta_0, \ch|_{\eta_0^\perp} \geq 0.
$$
\item There exists $\de_0>0$, so that
\begin{equation}
\label{650}
\ch|_{span\{\eta_0, Ker(\ch)\}^\perp}\geq \de_0 Id.
\end{equation}
Note that $Ker(\ch)=\{0\}$ is allowed.
\item There is $\xi_0\in Ker(\ch)^\perp, \xi_0 \neq 0$, so that
\begin{equation}
\label{600}
\dpr{\ch^{-1} \xi_0}{\xi_0}<0.
\end{equation}
\end{itemize}
Then,
$$
\ch|_{\xi_0^\perp}\geq 0.
$$
\end{lemma}
\begin{proof}(Lemma \ref{lina})
Without loss of generality, we may assume that $\|\xi_0\|=\|\eta_0\|=1$. Denote the positive invariant subspace of $\ch$ by $X_+$. That is, $X_+:=span\{\eta_0, Ker(\ch)\}^\perp$. Note $\ch|_{X_+}\geq \de_0$.
We take  $z\in \xi_0^\perp$ in the form
\begin{equation}
\label{620}
z=\eta_0+\psi_0+\psi, \psi_0\in Ker(\ch), \psi\in X_+.
\end{equation}
Clearly, since $\xi_0\in Ker(\ch)^\perp$,
$$
0=\dpr{z}{\xi_0}=\dpr{\eta_0}{\xi_0}+\dpr{\psi_0}{\xi_0}+\dpr{\psi}{\xi_0}=\dpr{\eta_0}{\xi_0}+\dpr{\psi}{\xi_0}
$$
Denote $\al=\dpr{\eta_0}{\xi_0}$, so that $\dpr{\psi}{\xi_0}=-\al$. Note that $\al\neq 0$, since otherwise, $\xi_0\in \eta_0^\perp$ and hence $\dpr{\ch^{-1} \xi_0}{\xi_0}\geq 0$, a contradiction with  \eqref{600}.

It clearly will suffice to prove that
$\dpr{\ch z}{z}\geq 0$, since $z$ is normalized so that $\dpr{z}{\eta_0}=1$, but otherwise an arbitrary element of $X_+$.  We have
$$
\dpr{\ch z}{z}=\dpr{\ch \eta_0}{\eta_0}+\dpr{\ch \psi}{\psi}=-\si^2+ \dpr{\ch \psi}{\psi}.
$$
Thus, it  remains to prove that $ \dpr{\ch \psi}{\psi}\geq \si^2$, whenever $\psi\in X_+: \dpr{\psi}{\xi_0}=-\al$.
To this end, consider the positive spectrum of $\ch$, that is $\si_+(\ch) :=\si(\ch|_{X_+})$. Note that by \eqref{650}, $\si_+(\ch)\subset [\de_0, \infty)$.
Consider the spectral decomposition of $\ch|_{X_+}$. We have a family of projections $E_\la in B(X_+)$, so that
for every $f\in X_+$ and every measurable  function $\chi$ on $\si_+(\ch)$, we have
$$
\chi(\ch) f=\int_{\si_+(\ch)} \chi(\la) dE_\la f.
$$
In particular,
$$
\psi=\int_{\si_+(\ch)}   dE_\la \psi.
$$
It follows that
\begin{eqnarray*}
|\al| =|\dpr{\psi}{\xi_0}|=|\int_{\si_+(\ch)}   d\dpr{E_\la \psi}{\xi_0}|\leq \int_{\si_+(\ch)}   d |\dpr{E_\la \psi}{\xi_0}|
\end{eqnarray*}
Now, by Cauchy-Schwartz's inequality and since $E_\la=E_\la^2$, we have
$$
|\dpr{E_\la \psi}{\xi_0}|=|\dpr{E_\la \psi}{E_\la \xi_0}|\leq \|E_\la \psi\| \|E_\la \xi_0\|=\sqrt{\dpr{E_\la \psi}{\psi}}
\sqrt{\dpr{E_\la \xi_0}{\xi_0}}
$$
Thus, by Cauchy-Schwartz and the properties of the spectral decomposition, we have
\begin{eqnarray*}
|\al| &\leq & \left( \int_{\si_+(\ch)}  \la d \dpr{E_\la \psi}{\psi}\right)^{1/2}
\left( \int_{\si_+(\ch)} \f{1}{\la} d \dpr{E_\la \xi_0}{\xi_0}\right)^{1/2} = \\
&=& \sqrt{\dpr{\ch \psi}{\psi}}
\sqrt{\dpr{\ch^{-1} P_{X_+}[\xi_0]}{P_{X_+}[\xi_0]}}.
\end{eqnarray*}
Note that since we required $\xi_0\perp Ker(\ch)$, we have that
$$
P_{X_+}[\xi_0]=P_{\eta_0^\perp} \xi_0=\xi_0-\dpr{\xi_0}{\eta_0} \eta_0=\xi_0-\bar{\al} \eta_0
$$
Thus,
\begin{eqnarray*}
\dpr{\ch^{-1} P_{X_+}[\xi_0]}{P_{X_+}[\xi_0]} &=& \dpr{\ch^{-1}[\xi_0-\bar{\al} \eta_0]}{[\xi_0-\bar{\al} \eta_0]}=\\
&=&
\dpr{\ch^{-1} \xi_0}{\xi_0} - 2\Re (\al \dpr{\ch^{-1} \xi_0}{\eta_0}) + |\al|^2 \dpr{\ch^{-1} \eta_0}{\eta_0}\leq \\
&=&  2\Re (\al \bar{\al}  \si^{-2}) - |\al|^2 \si^{-2} = \f{|\al|^2}{\si^2},
\end{eqnarray*}
where we have used the crucial inequality \eqref{600}. Plugging this result back in the inequality for $|\al|$, we have
$$
|\al|\leq \sqrt{\dpr{\ch \psi}{\psi}}  \f{|\al|}{\si}.
$$
Taking into account that $\al\neq 0$, we conclude $\dpr{\ch \psi}{\psi}\geq \si^2$, as required.

\end{proof}

   \subsection{ Positivity on a finite co-dimension  subspace}
 \label{sec:3.2}
In this section, we generalize the co-dimension one result to general co-dimensions.  \begin{theorem}
\label{theo:200}
Let $\cl$ has $k, k\geq 1$ negative eigenvalues, with $Z_0=span\{\eta_1, \ldots, \eta_k\}$. Assume that for some $\de_0>0$ , we have
$$
\ch|_{span\{Z, Ker(\ch)\}}\geq \de_0
$$
and for   some subspace $Z: dim(Z)\geq k, Z\perp Ker(\ch)$, $\ch$ satisfies
\begin{equation}
\label{z:300}
\ch^{-1}|_{Z}\leq -\de_0.
\end{equation}
Then,
$$
\ch|_{Z^\perp}\geq 0.
$$
\end{theorem}
{\bf Note:} Clearly, in the applications, one would like to apply Theorem \ref{theo:200} for
subspaces $Z$ with minimal dimension, that is $Z: dim(Z)=k$. This is allowed in the current formulation. We however prefer to state it with inequality for technical reasons, to be discussed below.
\begin{proof}
The proof proceeds by an induction argument on $k$.

For $k=1$, $Z$ is one dimensional, hence $Z=span\{\xi_0\}$ for some $\xi_0$ and the result is exactly Lemma \ref{lina}. Assume that we have proved it for some $k$ and consider an operator $\ch$ with $k+1\geq 2$ eigenvalues, $\eta_1, \ldots, \eta_{k+1}$ and $Z:dim(Z)\geq k+1$ is subspace.

Take an arbitrary element $0\neq \tilde{z}\in Z^\perp$. We claim that there exists an element $\tilde{\eta}: \|\tilde{\eta}\|=1$, with $\tilde{\eta}\in span\{\eta_1, \eta_2\}\subset   Z_0$ and $\tilde{z}\perp \tilde{\eta}$. Indeed, either $\tilde{z}\perp \eta_1$ or $\tilde{z}\perp \eta_2$,  in either of which cases we are fine   or $\tilde{z}$ is perpendicular to some linear combination of $\eta_1, \eta_2\subset Z_0$, say $\tilde{\eta}=a\eta_1+b \eta_2$, where $a^2+b^2=1$ by the normalization.
Consider the projection operator on to $\tilde{\eta}^\perp$ given by
$$
P_{\tilde{\eta}^\perp} f=f- \dpr{f}{\tilde{\eta}} \tilde{\eta}
$$
Consider the self-adjoint operator $\tilde{\ch}:=P_{\tilde{\eta}^\perp} \ch P_{\tilde{\eta}^\perp}$ and the subspace $\tilde{Z}:=P_{\tilde{\eta}^\perp} Z$. It is clear that $dim(\tilde{Z})\geq k$, since we project away at most one dimension.

We claim that $\tilde{\ch}$ has at most $k$ eigenvalues. This is completely obvious if for example $\tilde{\eta}=\eta_1$, since then $\tilde{\ch}=\ch|_{span\{\eta_2, \ldots, \eta_{k+1}\}}$, similar if
$\tilde{\eta}=\eta_2$. In the general case, we argue that the negative subspace of $\tilde{\ch}$ is
spanned by $\{-b\eta_1+a\eta_2, \eta_3, \ldots, \eta_{k+1}\}$. In fact, one can even explicitly compute  the $k$ negative eigenvalues of $\tilde{\ch}$ as follows\footnote{recall $a^2+b^2=1$}
\begin{equation}
\label{320}
 -\si_1^2 b^2 -\si_2^2 a^2, -\si_3^2, \ldots, -\si_{k+1}^2,
\end{equation}
where we have used the notation $-\si_j^2$ for the $j^{th}$ negative eigenvalue (i.e. $\ch \eta_j = -\si_j^2 \eta_j$). Note that the operator $\tilde{\ch}$ still has  $k+1$ non-positive eigenvalues, as it should! These are the $k$ listed in \eqref{320} and the newly generated zero eigenvalue, with eigenvector $\tilde{\eta}$. In this regard,  note $Ker(\tilde{\ch})=span\{Ker(\ch), \tilde{\eta}\}$.

We can now apply the induction hypothesis to $\tilde{\ch}$ and $\tilde{Z}$. Indeed, observe first that $\tilde{Z}\perp Ker(\ch)$, since $\tilde{Z}\subset Z$. In addition, $\tilde{Z}\perp \tilde{\eta}$ by construction. Thus, $\tilde{Z}\perp Ker(\tilde{\ch})$. In addition,   let $z\in \tilde{Z}\subset Z$. Then,
$z\perp \tilde{\eta}$ and
$$
\dpr{\tilde{\ch}^{-1} z}{z}=\dpr{\ch^{-1} z}{z}\leq -\de_0\|z\|^2,
$$
by the requirement \eqref{z:300} for $\ch$ and $Z$. This is \eqref{z:300} for the pair $\tilde{\ch}$ and $\tilde{Z}$. From the induction step, we conclude
$$
\tilde{\ch}|_{\tilde{Z}^\perp}\geq 0.
$$
In particular, for the arbitrary element $\tilde{z}\in Z^\perp$ that we have started with, we had $\tilde{z}\perp \tilde{\eta}$ and hence $\tilde{z}=P_{\tilde{\eta}^\perp} \tilde{z}$. We conclude
$$
 \dpr{\ch \tilde{z}}{\tilde{z}}=\dpr{\tilde{\ch} \tilde{z}}{\tilde{z}}\geq 0,
$$
which finishes the proof of the induction step and hence Theorem \ref{theo:200}.
\end{proof}

\section{Spectral stability   for the   periodic waves in the quadratic case}
\label{sec:4}
In this section, we consider the stability of the waves constructed in Section \ref{sec:3}.
Our main result is
\begin{theorem}
\label{theo:50}
The waves described in \eqref{2.5} are spectrally stable for all wave speeds $c>0$.
\end{theorem}
Before we proceed with the proof of Theorem \ref{theo:50}, we would like to discuss  an interesting limiting case. To that end,  take $\ka\to 1-$ in the solution $\Phi$ \eqref{2.5} and  then apply the  transformation  \eqref{80}.  We obtain the so-called parabolic peakons  in the form
\begin{equation}
\label{40}
\lim_{\ka\to 1-} \Phi_\ka (\eta(\xi))= \vp(\xi)=\f{\xi^2}{6}-\f{c}{2}, c = \f{L^2}{9}, \ \ -L<\xi < L.
\end{equation}
This construction goes back to at least the late 70's and it was revisited in several publications, \cite{Ostrov, GOSS, Stepan, Boyd, GHJ}. In fact, the authors in \cite{GHJ} placed special emphasis of these explicit solutions and asked about further properties of these simple solutions.
It is    easy to directly check in \eqref{30} (see also \cite{GHJ}),  that  the function displayed in \eqref{40}
provides a solution to \eqref{30} in $(-L,L)$.  Note that this choice of $c$ ensures that $\vp(-L)=\vp(L)$. However, while  the $2L$  periodization of $\vp$ is clearly  a continuous function, it  is not a differentiable function at $\pm L$. Indeed, we have
$$
\lim_{\xi \to L_-} \vp'(\xi)=\f{L}{3}, \ \ \lim_{\xi \to L_+} \vp'(\xi)=-\f{L}{3},
$$
which clearly do not match.  Thus,
one obtains {\it a peakon  solution}, with corner crests at all points $(2k+1) L, k \in {\mathbf N}$.
It would be interesting to see whether (an appropriate notion of) stability  holds for these waves. Note however, that if one does not impose appropriate
periodicity assumptions on $z$, one finds that \eqref{70} does  in fact have solution for some $\mu>0$ and $z\in C^2(-L,L)$, see Section \ref{sec:6} for details.

Over the course of the next few sections, we give the proof of Theorem \ref{theo:50}.
\subsection{Proof of Theorem \ref{theo:50}: preliminaries}
According to the derivation of \eqref{205}, we will show that there does not exists $Z\in L^2[-M,M]$ and
$\la: \Re\la>0$, so that \eqref{205} holds.  That is,  the wave $\vp(\xi)=\Phi(\eta)$  is stable. In fact, we will show that there does not exists $\mu: \Re\mu\neq 0$ so that
\begin{equation}
\label{520}
L[Z]=-c^2 Z_{\eta\eta} - c Z +  \Phi Z =  \mu Z_\eta.
\end{equation}
That is, we will show that the spectrum of the linearized operator is on the imaginary axis.
  We have that  the operator $L$ from \eqref{400} has two negative eigenvalues, an eigenvalue at zero, all simple, while the rest of the spectrum is contained in $(\ka, \infty)$, for some positive  $\ka>0$. That is, to introduce some notations,
$$
\left\{
\begin{array}{l}
\sigma(L)=\{-\si_0^2\}\cup \{-\si_1^2\}\cup \{0\}\cup \si_{+}(L)\\
L \chi_0=-\si_0^2 \chi_0, L \chi_1=-\si_1^2 \chi_1, L[\Phi]=0.
\end{array}
\right.
$$
We shall need, towards the end, explicit formulas for these functions and eigenvalues. One can of course write them explicitly, according to \eqref{c:10}, \eqref{c:20} and \eqref{c:30}, but we will not do so for now. We shall need to observe couple of things - first, $\chi_0$ and $\Phi$ are even functions, while $\chi_1$ is an odd function, second - note the relation $\chi_0(x)=\Phi(x)-c_0(\ka)$. This is easily seen, if one compares  \eqref{c:30}  (which provides $\Phi$) and \eqref{c:10}, which describes $\chi_0$. Note specifically however that $c_0(\ka)\neq c$, the speed of the wave.

We argue by contradiction. Assuming that there is a solution $Z$ of \eqref{520} (with some $\mu: \Re \mu\neq 0$), we establish some properties. These are later used to obtain a contradiction.

We take a dot product of \eqref{520} with $1$. We obtain
$$
\dpr{1}{L[Z]}=\mu \dpr{Z'}{1}=0,
$$
whence $0=\dpr{L[1]}{Z}=\dpr{\Phi-c}{Z}$, which implies $Z\perp \Phi - c$. Next, take a dot product of \eqref{520} with $\Phi$. We obtain
$$
\mu \dpr{Z'}{\Phi}=\dpr{L[Z]}{\Phi}=\dpr{Z}{L[\Phi]}=0,
$$
whence $Z\perp \Phi'$.

Finally, take a dot product of \eqref{520} with $\chi_0$, the eigenfunction corresponding to the lowest eigenvalue.
We have
\begin{equation}
\label{z:400}
\mu \dpr{Z'}{\chi_0}=\dpr{L[Z]}{\chi_0}=\dpr{Z}{L[\chi_0]}=-\si_0^2 \dpr{Z}{\chi_0}
\end{equation}
On the other hand, recalling that $\chi_0'=\Phi'$, we conclude
\begin{equation}
\label{z:510}
\dpr{Z'}{\chi_0}=-\dpr{Z}{\chi_0'}=-\dpr{Z}{\Phi'}=0.
\end{equation}
Combining \eqref{z:400} and \eqref{z:510}, we obtain $ \dpr{Z}{\chi_0}=0$. We have shown that $Z\perp \Phi-c_0(\ka)$, $Z\perp \Phi - c$, where $c_0(\ka)\neq c=c(\ka)$. In addition, $Z\perp \Phi'$.  It follows immediately that
\begin{equation}
\label{530}
Z\perp span\{1, \Phi, \Phi'\},
\end{equation}
 Having this information will allows us to rule out oscillatory/complex instabilities.
\subsection{Proof of Theorem \ref{theo:50}: ruling out complex instabilities}
Assuming that the eigenvalue problem \eqref{520} has a solution $\mu=\mu_1+i \mu_2, \mu_1\neq 0, \mu_2\neq 0$ and $Z=u+i v$ ($(u,v)\neq (0,0)$ real valued functions), we reach a contradiction. Indeed, with this notations, \eqref{520} is equivalent to the system
\begin{equation}
\label{540}
\left|
\begin{array}{l}
L u= \mu_1 u'-\mu_2 v' \\
L v = \mu_2 u'+\mu_1 v '
\end{array}
\right.
\end{equation}
 Next, taking into account that the operator $L$ preserves the parity of the function, we further split $u=u_1+ u_2, v=v_1+v_2$, where $u_1, v_1$ are even functions and $u_2, v_2$ are odd functions.  Projecting \eqref{540} in even and odd parts, we arrive at
\begin{equation}
\label{550}
\left|
\begin{array}{l}
L u_1= \mu_1 u_2'-\mu_2 v_2' \\
L u_2 = \mu_1 u_1'-\mu_2 v_1 '\\
L v_1=\mu_2 u_2'+\mu_1 v_2' \\
L v_2= \mu_2 u_1'+\mu_1 v_1'
\end{array}
\right.
\end{equation}
We have by the self-adjointness of $L$ (and the reality of all functions involved)
\begin{equation}
\label{560}
\mu_1 \dpr{u_2'}{v_1}-\mu_2 \dpr{v_2'}{v_1}=\dpr{L u_1}{v_1}=\dpr{L v_1}{u_1}=\mu_2\dpr{u_2'}{u_1}+\mu_1\dpr{v_2'}{u_1}.
\end{equation}
Similarly,
\begin{equation}
\label{570}
\mu_1 \dpr{u_1'}{v_2} - \mu_2 \dpr{v_1'}{v_2}=\dpr{L u_2}{v_2}=\dpr{L v_2}{u_2}=\mu_2 \dpr{u_1'}{u_2}+\mu_1\dpr{v_1'}{u_2}.
\end{equation}
Adding \eqref{560} and \eqref{570} yields
$$
\mu_1(\dpr{u_2'}{v_1}+\dpr{u_1'}{v_2})=\mu_1(\dpr{v_2'}{u_1}+\dpr{v_1'}{u_2}).
$$
Taking into account $\mu_1\neq 0$ and $\dpr{v_1'}{u_2}=-\dpr{u_2'}{v_1}$ and $\dpr{u_1'}{v_2}=-(\dpr{v_2'}{u_1}$, we conclude that
\begin{equation}
\label{580}
\dpr{u_2'}{v_1}=\dpr{v_2'}{u_1}.
\end{equation}
Subtracting \eqref{570} from \eqref{560} yields
$$
\mu_2(\dpr{v_1'}{v_2}-\dpr{v_2'}{v_1})=\mu_2(\dpr{u_2'}{u_1}-\dpr{u_1'}{u_2}).
$$
Since $\mu_2\neq 0$, we have
\begin{equation}
\label{590}
\dpr{v_1'}{v_2}=\dpr{u_2'}{u_1}.
\end{equation}
Evaluating
\begin{equation}
\label{528}
\begin{array}{l l}
\dpr{L u_1}{u_1}+ \dpr{L v_1}{v_1}   & =   \mu_1\dpr{u_2'}{u_1}-\mu_2\dpr{v_2'}{u_1}+\mu_2\dpr{u_2'}{v_1}+
\mu_1 \dpr{v_2'}{v_1} = \\
& = \mu_1(\dpr{u_2'}{u_1}+ \dpr{v_2'}{v_1})+\mu_2(\dpr{u_2'}{v_1} - \dpr{v_2'}{u_1})=0,
\end{array}
\end{equation}
where in the last line, we have used \eqref{580} and \eqref{590}.
Now, the condition \eqref{530} implies in particular that
\begin{equation}
\label{525}
u_1, v_1 \perp span\{\chi_0, \Phi, \chi_1\}.
\end{equation}
Indeed, since $u_1, v_1$ are even and $\chi_1$ is odd, we have $u_1, v_1 \perp \chi_1$. On the other hand, by \eqref{530}, we know that $u_1+u_2=\Re Z \perp span\{\chi_0, \Phi\}=span\{1, \Phi\}$. Observe that by parity considerations,
$u_2\perp span\{\chi_0, \Phi\}=span\{1, \Phi\}$.
 Thus, it must be that $u_1 \perp span\{\chi_0, \Phi\}$. Similarly,
$v_1 \perp span\{\chi_0, \Phi\}$, which is \eqref{525}.

By  \eqref{525} and the structure of the spectrum of $L$,  it follows that $u_1, v_1$ lie in the positive subspace of $L$, whence
$$
0=\dpr{L u_1}{u_1}+ \dpr{L v_1}{v_1}\geq \de(\|u_1\|^2+\|v_1\|^2).
$$
Clearly, this implies $u_1=v_1=0$. Going back to \eqref{550}, we obtain, for some constants $c_1, c_2$,
$$
\left|
\begin{array}{l}
\mu_1 u_2 - \mu_2 v_2=c_1\\
\mu_2 u_2+\mu_1 v_2=c_2
\end{array}
\right.
$$
Again, since $\mu_1\neq 0, \mu_2\neq 0$ implies that $u_1=C_1, u_2=C_2$ for some constants $C_1, C_2$. But then
$$
C_1 L[1]=0, C_2 L[1]=0,
$$
which combined with $L[1]=\Phi-c\neq 0$ implies  $C_1=C_2=0$. Thus, we have reached the zero solution, a contradiction.

\subsection{Proof of Theorem \ref{theo:50}: ruling out real instabilities}
Having this lemma in mind, we can rule out real instabilities for the eigenvalue problem \eqref{520}, provided, we know that
\begin{equation}
\label{700}
\dpr{L^{-1}[\Phi']}{[\Phi']}<0.
\end{equation}
 Indeed, assume the validity of \eqref{700} and assume, for a contradiction, that for some $\mu>0$ and $Z$ real-valued, we have \eqref{520}. As it was established in \eqref{530}, we have that $Z\perp \chi_0, \Phi, \Phi'$. Split   in even and odd functions as before:  $Z=u+v$, where $u$ is even and $v$ is odd. The eigenvalue problem reduces to
$$
\left|
\begin{array}{l}
L u = \mu v' \\
L v=\mu u'
\end{array}
\right.
$$
Taking dot products with $u$ and $v$ respectively and adding yields
$$
\dpr{Lu}{u}+\dpr{L v}{v}=\mu(\dpr{v'}{u}+\dpr{u'}{v})=0.
$$
As before, $u\perp \chi_0, \chi_1, \Phi$ and hence $\dpr{L u}{u}>0$ (unless $u=0$, in which case, the contradiction is obvious  right away). Thus, it follows that $\dpr{L v}{v}<0$.

We will show that this last inequality leads to a contradiction as well.
Indeed, consider the Hilbert space
$X_0:=L^2_{odd}$, which is clearly an invariant subspace for $L$. We are in a position to apply Lemma \ref{lina} to the operator $\ch:=L$, acting on $X_0=L^2_{odd}$ with $\eta_0:=\chi_1$. Clearly, $\ch|_{L^2_{odd}\cap \{\chi_1\}^\perp}\geq 0$. Take $\xi_0:=\Phi'\in X_0$. Clearly,
$\xi_0 \perp Ker(L)=span\{\Phi\}$. In addition, we assume \eqref{700}. Thus, by the conclusion of Lemma \ref{lina}, we will have
\begin{equation}
\label{710}
L|_{L^2_{odd}\cap \{\Phi'\}^\perp}\geq 0.
\end{equation}
On the other hand, $Z=u+v \perp \Phi'$, so it follows that $v\perp \Phi'$ (since $u\perp \Phi'$ by parity).  Of course $v\in L^2_{odd}$, while $\dpr{L v}{v}<0$, a contradiction with \eqref{710}.

\subsection{Computing $\dpr{L^{-1} [\Phi']}{[\Phi']}$}
Due to the rescaling properties, it will suffice to work with the following operators and functions
\begin{eqnarray*}
\be(k) &:= & 1+k^2+\sqrt{1-k^2+k^4}; \\
\ga(k) &:=& \f{6k^2}{\be(k)} -2 \be(k); \\
\tilde{\Phi}(y,k) &:= & 6 k^2 sn^2(y,k) - \f{6k^2}{\be(k)};\\
\tilde{L} &:= & -\p_{yy}+\ga(k)+\tilde{\Phi} \ \ \ 0\leq y\leq 2 K(k).
\end{eqnarray*}
We have $\tilde{L}[\tilde{\Phi}]=0$.
 Our goal is to show \eqref{700}, which is equivalent to
\begin{equation}
\label{800}
\dpr{\tilde{L}^{-1}[\tilde{\Phi}']}{[\tilde{\Phi}']}<0.
\end{equation}
In order to compute $\tilde{L}^{-1}$, we need to construct its Green function. This is achieved by finding another non-trivial solution $\tilde{\Psi}: \tilde{L}[\tilde{\Psi}]=0$. There are methods for constructing these, roughly by looking at a variation of constants  formula like
$$
\tilde{\Psi}(x)= \tilde{\Phi}(x)\int^{x} \f{1}{\tilde{\Phi}^2(y)} dy.
$$
It turns out that this, while possible formally, leads to some issues because of the vanishing in the denominator of the integral. So, we just postulate
$$
\tilde{\Psi}(x):=\tilde{\Phi}(x)\int_0^{x} \tilde{\Phi}(y) dy - 3 \tilde{\Phi}'(x).
$$
Let us verify that this indeed satisfies $\tilde{L}[\tilde{\Psi}]=0$. By differentiating this identity, we obtain $\tilde{L}[\tilde{\Phi}']=-\tilde{\Phi}' \tilde{\Phi}$. We have
\begin{eqnarray*}
\tilde{L}[\tilde{\Psi}]=\left(\tilde{L}[\tilde{\Phi}] \int_0^{x} \tilde{\Phi}(y) dy\right)  -2 \tilde{\Phi}' \tilde{\Phi}-
\tilde{\Phi} \tilde{\Phi}'-3 \tilde{L}[\tilde{\Phi}']=-3 \tilde{\Phi}' \tilde{\Phi}+3 \tilde{\Phi}' \tilde{\Phi}=0.
\end{eqnarray*}
where we have used $\tilde{L}[\tilde{\Phi}]=0$. It is not hard to check that the function $\tilde{\Psi}$ is odd function in $[-K(k), K(k)]$.
Define the ($x$ independent)  Wronskian
$$
W[k]=\tilde{\Psi}'(x,k)\tilde{\Phi}(x,k)- \tilde{\Phi}'(x,k)\tilde{\Psi}(x,k).
$$

One can now construct the inverse of $\tilde{L}$ as follows. Namely,  for a function $f\perp Ker[\tilde{L}]=span[\tilde{\Phi}]$,
$$
L^{-1} f = \f{1}{W[k]}\left(\tilde{\Phi}(x) \int_{-K(k)} ^x \tilde{\Psi}(y) f(y) dy -  \tilde{\Psi}(x) \int_{-K(k)}^x \tilde{\Phi}(y) f(y) dy  \right)+ C_f \tilde{\Psi}(x),
$$
where the constant $C_f$ is chosen so that $L^{-1}[f]$ is $2 K(k)$ periodic. In our case, $f=\tilde{\Phi}'$ a we obtain the following formula for  $C_{\tilde{\Phi}'}$
$$
C_{\tilde{\Phi}'}= - \f{ \tilde{\Phi}(K(k))}{W[k]\tilde{\Psi}(K(k))}  \int_0^{K(k)} \tilde{\Psi}(y)\tilde{\Phi}'(y) dy.
$$
Thus, we have, after some elementary integration by parts and using the periodicity of $\tilde{\Phi}$
\begin{eqnarray*}
\dpr{\tilde{L}^{-1}[\tilde{\Phi}']}{[\tilde{\Phi}']} &=&
\f{2}{W[k]}\left( \tilde{\Phi}^2(K(k)) \int_0^{K(k)} \tilde{\Psi}(y) \tilde{\Phi}'(y) dy   - \int_0^{K(k)} \tilde{\Phi}^2(y)\tilde{\Psi}(y) \tilde{\Phi}'(y) dy
  \right) - \\
&-& \f{2}{W[k]}   \f{  \tilde{\Phi}(K(k))}{ \tilde{\Psi}(K(k))}  \left(\int_0^{ K(k)} \tilde{\Psi}(y)\tilde{\Phi}'(y) dy\right)^2.
\end{eqnarray*}
Clearly, matters reduce to the computation of the following   integrals
$$
 \int_0^{K(k)} \tilde{\Psi}(y) \tilde{\Phi}'(y) dy, \ \ \int_0^{K(k)} \tilde{\Phi}^2(y)\tilde{\Psi}(y) \tilde{\Phi}'(y) dy.
$$
We have
\begin{eqnarray*}
I_1(k) &:=  & \int_0^{K(k)} \tilde{\Psi}(y) \tilde{\Phi}'(y) dy= -3 \int_0^{K(k)}  (\tilde{\Phi}'(y))^2  dy  - \f{1}{2} \int_0^{K(k)} (\tilde{\Phi}(y))^3 dy+ \\
&+& \f{\tilde{\Phi}^2(K(k))}{2} \int_0^{K(k)} \tilde{\Phi}(y) dy.
\end{eqnarray*}
On the other hand,
\begin{eqnarray*}
I_2(k) &:=  &  \int_0^{K(k)} \tilde{\Phi}^2(y)\tilde{\Psi}(y) \tilde{\Phi}'(y) dy=-3
\int_0^{K(k)} \tilde{\Phi}^2(y)  (\tilde{\Phi}'(y))^2 dy - \f{1}{4} \int_0^{K(k)} \tilde{\Phi}^5(y) dy + \\
&+& \f{\tilde{\Phi}^4(K(k))}{4} \int_0^{K(k)} \tilde{\Phi}(y) dy.
\end{eqnarray*}
Thus, we may write now,
 \begin{eqnarray*}
\dpr{\tilde{L}^{-1}[\tilde{\Phi}']}{[\tilde{\Phi}']} &=& \f{2}{W(k)}\left(\tilde{\Phi}^2(K(k)) I_1(k) - I_2(k)- \f{  \tilde{\Phi}(K(k))}{ \tilde{\Psi}(K(k))} (I_1(k))^2 \right).
\end{eqnarray*}
This last expression is clearly a function of $k$ only. While one can in principle compute this expression explicitly by hand, we have used Mathematica for a symbolic integration. Here are some of the formulas that we have found.
The Wronskian is a positive function  on $[0,1]$, given by
$$
W[k]=16 (2 (\sqrt{k^4-k^2+1}-1)+k^4(2
   \sqrt{k^4-k^2+1}+3-2k^2)-2k^2 \sqrt{k^4-k^2+1}+3k^2).
$$
 while $h(k)=\tilde{\Phi}^2(K(k)) I_1(k) - I_2(k)- \f{  \tilde{\Phi}(K(k))}{ \tilde{\Psi}(K(k))} (I_1(k))^2=\f{A(k)}{B(k)}$ where
\begin{eqnarray*}
B(k) &=& (k^2+\sqrt{k^4-k^2+1}+1)^4
   ((k^2+\sqrt{k^4-k^2+1}+1) E(k)-
(\sqrt{k^4-k^2+1}+1) K(k))
\end{eqnarray*}
  $A(k)$ is a little complicated to display, so we don't provide the explicit formula  here.
We can however plot the graph of this function for $k\in (0,1)$, see Figure \ref{fig1} below.  Since obviously the function is negative for all values of $k\in (0,1)$, we conclude that \eqref{800} holds and thus, Theorem \ref{theo:50} is proved in full.
\begin{figure}[h]
\centering
\includegraphics[width=8cm,height=6cm]{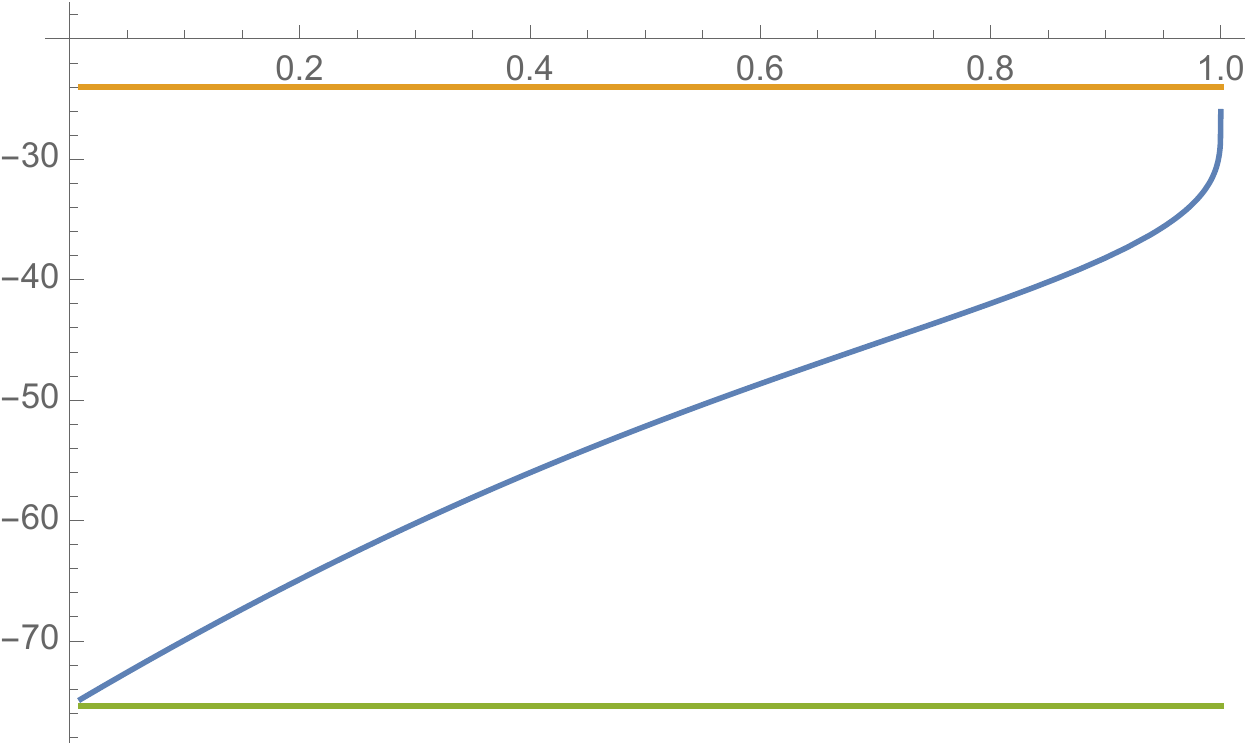}
\caption{The function $2 h(k)/W(k)$ in blue. The yellow line is at $y=-24=\lim_{k\to 1-} \frac{2 h(k)}{W(k)}$, while the red line is at $y=-24\pi=\lim_{k\to 0+} \frac{2 h(k)}{W(k)}$}
\label{fig1}
\end{figure}
            \vspace{1cm}

 \section{Spectral stability for the periodic waves in the cubic model}
\label{sec:10}
Our main result for the cubic model is the following.
\begin{theorem}
\label{theo:100}
The waves constructed in \eqref{per14} and \eqref{per15} are spectrally stable with respect to co-periodic  perturbations.
\end{theorem}

The proof of Theorem \ref{theo:100} follows path similar to the proof of Theorem \ref{theo:50}.
We first rule out complex instabilities. For the real instabilities, we need to use
Theorem \ref{theo:200} for specific co-dimension two subspace $Z: Z\perp Ker(L)$.  Finally, we need to verify that $L^{-1}|_Z<0$, which reduces to verifying the same type of quantity as before.

Based on the information for the operator $L$ in \eqref{cc2}, we have the following spectral picture
$$
\left\{
\begin{array}{l}
\sigma(L)=\{-\si_0^2\}\cup \{-\si_1^2\}\cup   \{0\}\cup \si_{+}(L)\\
L \chi_0=-\si_0^2 \chi_0,  L\chi_1 = -\si_1^2 \chi_1 L[\Phi]=0.
\end{array}
\right.
$$
Note that $\chi_0, \chi_1$ are even functions, while $\Phi$ is an odd function. Next, we
are considering the eigenvalue problem \eqref{z:65}. Taking dot product with the constant $1$ yields
$$
\dpr{L[Z]}{1}=\mu \dpr{Z'}{1}=0.
$$
Thus, $\dpr{Z}{L[1]}=\dpr{L[Z]}{1}=0$, whence $Z\perp L[1]=\Phi^2-c$.  Taking dot product of \eqref{z:65} with $\Phi$, we obtain
$$
\mu\dpr{Z'}{\Phi}=\dpr{L[Z]}{\Phi}=\dpr{Z}{L[\Phi]}=0,
$$
whence $Z\perp \Phi'$.

\subsection{Proof of Theorem \ref{theo:100}: Ruling out complex instabilities}
We now rule out complex instabilities. We proceed in the same way as in the derivation of \eqref{540} and subsequently \eqref{550}, \eqref{580}, \eqref{590}.  Instead of verifying the quantity \eqref{528} however\footnote{which is by the way still valid}, we now compute
\begin{eqnarray*}
\dpr{L u_2}{u_2}+\dpr{L v_2}{v_2} &=& \mu_1(\dpr{u_1'}{u_2}+\dpr{v_1'}{v_2})+
\mu_2(\dpr{u_1'}{v_2}-\dpr{v_1'}{u_2})=\\
&=& \mu_1(\dpr{u_1'}{u_2}+\dpr{u_2'}{u_1})+\mu_2(\dpr{u_1'}{v_2}+\dpr{v_2'}{u_1})=0.
\end{eqnarray*}
where   we have used \eqref{580} and \eqref{590}. But now, $u_2, v_2$ are both odd functions and as such, project on the non-negative subspace of $L$ (recall that the negative subspace is spanned by $\chi_0, \chi_1$, both even functions). Thus, it follows that $\dpr{L u_2}{u_2}=0$ and
$\dpr{L v_2}{v_2}=0$, whence
$$
u_2=C_1 \Phi, v_2=C_2 \Phi.
$$
Putting this back in \eqref{550} yields in particular
$$
\left|
\begin{array}{l}
\mu_1 u_1'-\mu_2 v_1'=0 \\
\mu_2 u_1'+\mu_1 v_1'=0
\end{array}
\right.
$$
Since $\mu_1\neq 0, \mu_2\neq 0$, it follows that $u_1=A_1, v_1=B_1$. But then again from \eqref{550}, we obtain
$$
\left|
\begin{array}{l}
A_1 L[1]=(\mu_1 C_1-\mu_2 C_2)\Phi' \\
B_1 L[1]=(\mu_2 C_1+\mu_1 C_2) \Phi'.
\end{array}
\right.
$$
Since $L[1]=\Phi^2-c\neq \Phi'$, we have $A_1=B_1=0$ and $C_1=C_2=0$, which is the zero solution for the spectral problem, a contradiction.

\subsection{Proof of Theorem \ref{theo:100}: Ruling out real instabilities}
In this section, we rule out the real instabilities. Assume for a contradiction that there is $\mu>0$, so that the eigenvalue problem \eqref{z:65} has a solution $Z=u+v$, where $u$ is even and $v$ is an odd function. We have then
\begin{equation}
\label{n:30}
\left|
\begin{array}{l}
L u = \mu v' \\
L v=\mu u'
\end{array}
\right.
\end{equation}
Taking dot products with $u$ and $v$ respectively and adding
\begin{equation}
\label{n:40}
\dpr{Lu}{u}+\dpr{L v}{v}= \mu(\dpr{v'}{u}+\dpr{u'}{v})=0.
\end{equation}
Again, the odd function $v$ projects over the non-negative subspace of $L$ only. Thus,
$\dpr{L v}{v}\geq 0$. Assume first that $\dpr{L v}{v}=0$. It follows that $v=c_0 \Phi$ for some constant $c_0$. It follows from \eqref{n:30} that
$$
\mu u'=L[c_0 \Phi]=0,
$$
whence $u=c_1$. But then, from the other equation in \eqref{n:30},
$$
c_1 L[1]=\mu c_0 \Phi'.
$$
Again $L[1]=\Phi^2-c\neq \Phi'$ and this is only possible if $c_1=c_0=0$, a contradiction.

Thus, it must be that $\dpr{L v}{v}>0$. By \eqref{n:40}, this implies that $\dpr{Lu}{u}<0$.
We will show that this leads to a contradiction as well.

 We need the following simple lemma\footnote{The lemma is well-known, but we add its proof for completeness, since we don't have a direct reference for it.}.
\begin{lemma}
\label{le:simple}
Let $\cx$ be a real Hilbertian space, with $Z=span\{\eta_1, \eta_2\}\subset \cx$, where
$\eta_1, \eta_2: \eta_1\perp \eta_2$. Let $A$ be a self-adjoint operator on $\cx$. Then
$A|_Z<0$ if and only if the matrix
$$
D=(D_{ij})_{i,j=1,2}, \ \ D_{ij}=\dpr{A \eta_i}{\eta_j}.
$$
is negative definite. Equivalently, $D$ is negative definite if
$$
\dpr{A\eta_1}{\eta_1}<0, \det(D)>0.
$$
\end{lemma}

\begin{proof}
Assume  $A|_Z<0$. Immediately, $\dpr{A\eta_1}{\eta_1}<0$.  Let
$$
g(\la):=\dpr{A(\la \eta_1+\eta_2)}{\la\eta_1+\eta_2}=\la^2 \dpr{A\eta_1}{\eta_1}+
2 \la \dpr{A\eta_1}{\eta_2}+\dpr{A\eta_2}{\eta_2}.
$$
In particular,  $g(\la)<0$ for all $\la$.  Thus,  the quadratic function $\la\to g(\la)$ does not have real roots, that is
$$
\dpr{A\eta_1}{\eta_2}^2-\dpr{A\eta_1}{\eta_1} \dpr{A\eta_2}{\eta_2}<0.
$$
Since $\det(D)=-(|\dpr{A\eta_1}{\eta_2}|^2-\dpr{A\eta_1}{\eta_1} \dpr{A\eta_2}{\eta_2})>0$, we have shown one direction.

Conversely, assume that $D$ is negative definite. Then since $\det(D)>0$, it follows that $g(\la)=0$ does not have real solutions. This, paired with $\dpr{A\eta_1}{\eta_1}<0$ implies that $g(\la)<0$ for all real $\la$, which in turn means that $A|_Z<0$.
\end{proof}
We claim that in order to obtain a contradiction, it is enough to show
\begin{equation}
\label{m:10}
L^{-1}|_{span\{\Phi', \Phi^2-c\}}<0.
\end{equation}
Indeed, since we already have that $Z=u+v\perp \Phi'$ and $v\perp \Phi'$ (as an odd function), we conclude that $u\perp \Phi'$. Similarly, $Z=u+v\perp \Phi^2-c$, $v\perp \Phi^2-c$ (as an odd function), whence $u\perp \Phi^2-c$. Overall, $u\in span\{\Phi', \Phi^2-c\}^\perp$. On the other hand, if \eqref{m:10} holds, by Theorem \ref{theo:200}, we would conclude $L|_{span\{\Phi', \Phi^2-c\}^\perp}\geq 0$, so in particular $\dpr{Lu}{u}\geq 0$, a contradiction with \eqref{n:40}, which implies $\dpr{L u}{u}<0$. Thus, matters have been reduced to showing \eqref{m:10}.

By Lemma \ref{le:simple}, in order to prove \eqref{m:10},  it is enough to show that the  matrix
$$
D=\left(
\begin{array}{cc}
\dpr{L^{-1}[\Phi']}{\Phi'} & \dpr{L^{-1}[\Phi^2-c]}{\Phi'} \\
\dpr{L^{-1}[\Phi']}{\Phi^2-c} & \dpr{L^{-1}[\Phi^2-c]}{\Phi^2-c}
\end{array}
\right)
$$
is negative definite. We have (recalling $\Phi^2-c=L[1]$),
$$
\dpr{L^{-1}[\Phi']}{\Phi^2-c} =\dpr{\Phi'}{L^{-1}[\Phi^2-c]}= \dpr{\Phi'}{L^{-1}[L[1]]}=\dpr{\Phi'}{1}=0.
$$
Also,
$$
\dpr{L^{-1}[\Phi^2-c]}{\Phi^2-c} = \dpr{L^{-1}[L[1]]}{\Phi^2-c} = \dpr{1}{\Phi^2-c}<0,
$$
where in the last step, we have made use of \eqref{z:70}. Thus, the matrix $D$ is diagonal, with $D_{22}<0$. It would follow that $D$ is negative definite, if we can verify that $D_{11}<0$.
Thus, we have reduced  matters,   again,  to showing that
\begin{equation}
\label{z:10}
\dpr{L^{-1}[\Phi']}{\Phi'}<0.
\end{equation}

\subsection{Computing $\dpr{L^{-1}[\Phi']}{\Phi'}$}
By rescaling and from \eqref{cc2},  we may take $L$ as follows
\begin{eqnarray*}
L &=& -\p_y^2+ 2\ka^2 sn^2(y,k)-(1+\ka^2)\\
\Phi &=&  sn(y,k), -2K(\ka)\leq y\leq 2 K(\ka).
\end{eqnarray*}
As before, the first order of business is to construct the Green's function. We need a second function in the kernel of $L$, in addition of $\Phi$. Normally, we would take
$
\Psi(x)=\Phi(x) \int^x \f{1}{\Phi^2(y)} dy,
$
but note that the (definite) integral would be divergent over any interval containing zero, because of the quadratic singularity of $sn(y,k)$ at $y=0$. Instead, one integrates by parts and we come up with an equivalent expression, which is however well-defined. Namely, using that 
$$\f{1}{sn^2(x, \kappa)}=-\f{1}{dn(x, \kappa)}\f{\partial}{\partial x}\f{cn(x, \kappa)}{sn(x, \kappa)}$$
 and (formally) integrating by parts, we get  
\begin{equation}
\label{m:50}
\Psi(x)=-\f{cn(x, \kappa)}{dn(x, \kappa)}+\kappa^2 sn(x, \kappa)\int_{0}^{x}{\f{cn^2(y, \kappa)}{dn^2(y, \kappa)}}dy.
\end{equation}
This formula makes sense  - note the lack of singularity in the denominator. 
Interestingly, for fixed $\ka$, the function $x\to \Psi(x,\ka)$ turns out to be a periodic function in the basis interval $[-2K(\ka), 2 K(\ka)]$, but its derivative $x\to \p_x \Psi(x,\ka)$ is not periodic\footnote{This is complicates matters somewhat, when one construct the Green's function, but not in a major way} anymore at $\pm 2 K(\ka)$.

Define the Wronskian by the standard formula
$$
W[\kappa]=\Psi'(x, \ka)\Vp(x, \ka)-\Psi(x, \ka)\Vp'(x, \ka).
$$
Using Mathematica, we have found that $W[\ka]=1$, which confirms once again that the function  $\Psi$ constructed in \eqref{m:50} is another non-trivial solution of $L\Psi=0$.  We can now represent for any
$f\in L^2_{per}[-2K(\ka), 2K(\ka)]$,
$$
L^{-1} f (x)= \f{1}{2}\left( \Phi(x) \int_{-2K(\ka)}^x \Psi(y)  f(y) dy- \Psi(x) \int_{-2K(\ka)}^x \Phi(y)  f(y)  dy\right)+ C_f \Psi(x).
$$
where the constant $C_f$ is to be selected so that the function $x\to L^{-1} f(x)$ is periodic in
$[-2K(\ka), 2K(\ka)]$. 

For the function of interest, namely  $f=\Phi'$, we find that
$L^{-1} f(-2K(k))=L^{-1} f(2 K(\ka))$ for all values of $C$, but when we impose the condition
$\p_x  L^{-1} f|_{x=-2K(k)}= \p_x  L^{-1} f|_{x= 2K(k)}$, we come up with a condition for 
$C_{\Phi'}$ 
$$
C_{\Phi'}=\f{1}{4\Psi'(2K(\ka))} \int_{-2K(\ka)}^{2K(\ka)} \Psi(x) \Phi'(x) dx.
$$
Now that we have a proper formula for $L^{-1}[\Phi']$, we may compute
\begin{eqnarray*}
\dpr{L^{-1}[\Phi']}{\Phi'} &=&  \f{1}{2} \left[ \left\langle \Vp\int_{-2K(\kappa)}^{x}{\Psi(y)\Vp'(y)}dy, \Vp'\right\rangle-\left\langle \Psi(x)\int_{-2K(\kappa)}^{x}{\Vp(y)\Vp'(y)}dy, \Vp'\right\rangle\right]+ \\
&+& \f{1}{4\Psi'(2K(\ka))}  (\langle \Psi, \Vp'\rangle)^2.
\end{eqnarray*}
After integrating by parts, we get
 $$
 \left\langle \Vp\int_{-2K(\kappa)}^{x}{\Psi(y)\Vp'(y)}dy, \Vp'\right\rangle=-\f{1}{2}\int_{-2K(\kappa)}^{2K(\kappa)}{\Vp^2(x)\Vp'(x)\Psi(x)}dx
 $$
 and
 $$
 \left\langle \Psi(x)\int_{-2K(\kappa)}^{x}{\Vp(y)\Vp'(y)}dy, \Vp'\right\rangle=
 \f{1}{2}\int_{-2K(\kappa)}^{2K(\kappa)}{\Vp^2\Vp'\Psi}dx.
 $$
Putting it together in the formula for $\dpr{L^{-1}[\Phi']}{\Phi'}$, we obtain
\begin{equation}
\label{m:100}
\dpr{L^{-1}[\Phi']}{\Phi'}= -\f{1}{2} \int_{-2K(\kappa)}^{2K(\kappa)}{\Vp^2\Vp'\Psi}dx+\f{1}{4\Psi'(2K(\ka))}  \left( \int_{-2K(\kappa)}^{2K(\kappa)} \Psi \Vp' \right)^2.
\end{equation}
Clearly, this last expression is a function of $\ka$ only. Computing precisely the integrals by hand
is not easy, even though some of them result in simple expressions, for example 
$$
\int_{-2K(\kappa)}^{2K(\kappa)} \Psi \Vp' = - 2K(\ka).
$$
We use Mathematica for the symbolic integration\footnote{Even within Mathematika, we need to use the representation \eqref{m:50} to split and integrate in parts by hand, whenever possible, in order to obtain explicit formulas} to obtain the precise formulas. 
 Below, we provide the graph of the resulting function, from which it is clear that \eqref{z:10} is satisfied.
  \begin{figure}[h]
\centering
\includegraphics[width=8cm,height=6cm]{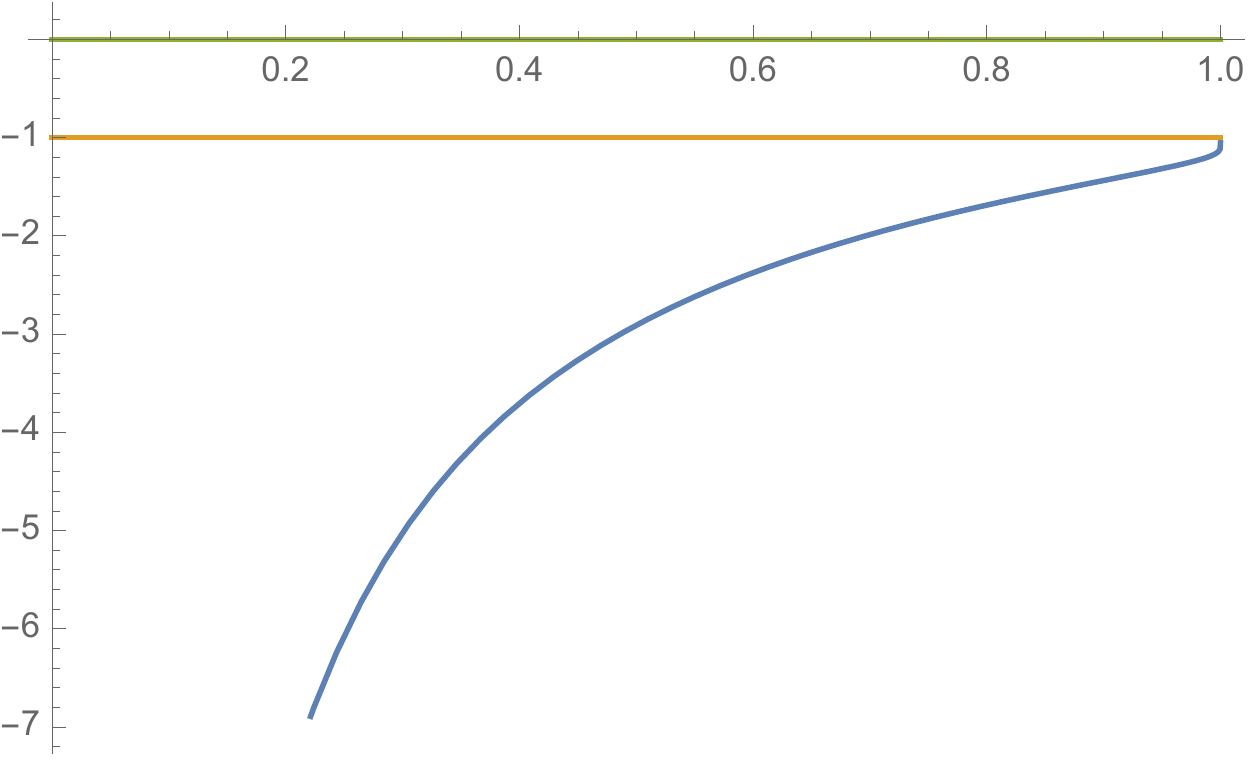}
\caption{The function $\ka\to \dpr{L^{-1}[\Phi']}{\Phi'}$ in blue. The orange line is  $y=-1=
\lim_{k\to 1-}  \dpr{L^{-1}[\Phi']}{\Phi'}$. }
\label{fig14}
\end{figure}

\section{The parabolic  peakons  are ``unstable'' in the absence of periodic boundary conditions}
\label{sec:6}
As we have discussed in the previous sections, the parabolic peakons may be considered as limit solutions of the arise as a limit of the stable snoidal solutions of the Ostrovsky model
introduced in \eqref{2.5}.  In that sense, one might be tempted to claim that they are
stable. We have however left  this question open. This is partially due to the difficulty of defining an appropriate boundary conditions for the corresponding linearized problem.
 In this section, we construct (almost explicitly)  solutions of \eqref{z:62}, which   lack the required periodic properties, but are nevertheless smooth inside $(-L,L)$ and satisfy the equation in classical sense. So, one can say that in a way the parabolic peakons are unstable, once the perturbations are allowed to have ``wild'' boundary conditions.
\begin{theorem}
\label{theo:10}
The parabolic peakons $\vp$, given by \eqref{40} are spectrally unstable for all $L>0$ in the sense that there exists $z\in C^2(-L,L)$ (in fact $z\in C^\infty(-L,L)$) so that \eqref{70} is satisfied in a classical sense for $\xi \in (-L,L)$.
\end{theorem}
{\bf Note:} As we establish below, the  $2L$ periodization of the function $z$ constructed here is not even continuous at $\pm L$, since
$$
\lim_{\xi \to -L+} z(\xi)\neq 0, \ \ \lim_{\xi \to L-} z(\xi)=0.
$$
Thus, the instability result claimed in Theorem \ref{theo:10} is not in the sense of
Definition \ref{defi:1}, but rather in a milder sense, see the precise  description above.   \begin{proof}
Our proof proceeds via the same approach, via the change of variables \eqref{80}.
  In this simple case, with the explicit formula \eqref{40}, we will be able however to explicitly calculate   $\xi=\xi(\eta)$ and its inverse. Indeed, taking derivative in $\eta$ in \eqref{80}, we obtain
\begin{equation}
\label{90}
\f{d\Xi}{d\eta}=1-\f{\Psi'(\eta)}{c}=1-\f{\vp(\xi)}{c}=\f{3}{2 L^2} (L^2-\xi^2).
\end{equation}
Treating this last equality as a separable ODE, we have $\f{d\xi}{L^2-\xi^2}=\f{3}{2 L^2} d\eta$. By integrating the ODE, we obtain the formula
$$
\xi=L \f{C e^{\f{e^{3 \eta}}{L}}-1}{C e^{\f{e^{3 \eta}}{L}}+1}.
$$
This gives us an arbitrary solution of \eqref{90}. We set $C=1$ as a particular solution, which gives us
\begin{equation}
\label{100}
\xi=\Xi(\eta)=L \f{e^{\f{e^{3 \eta}}{L}}-1}{e^{\f{e^{3 \eta}}{L}}+1}=L \tanh(3\eta/(2L)).
\end{equation}
Clearly, the function $\Xi:(-\infty, \infty)\to (-L,L)$ is one-to-one and it has an inverse function $\eta=\eta(\xi): (-L,L)\to (-\infty, \infty)$. Thus, in this case the interval $(-M,M)$, which we used to get from inverting the transformation $\xi=\Xi(\eta)$ is actually degenerate, in that $M=\infty$.
Using this last formula in \eqref{40}, we obtain a formula for $\Phi$. More precisely,
\begin{equation}
\label{200}
\Phi(\eta)=\vp(\Xi(\eta))= \f{L^2}{6} \tanh^2(3\eta/(2L))-\f{L^2}{18}=\f{L^2}{9}
\left(1-\f{3}{2}
\sech^2(3\eta/(2L))\right).
\end{equation}
Note that $-c+\Phi= - \f{L^2}{6} \sech^2(3\eta/(2L))$. We will show that for some $\mu>0$, the resulting eigenvalue problem
\begin{equation}
\label{210}
(-\f{L^2}{81} \p_{\eta\eta} - \f{1}{6} \sech^2(3\eta/(2L))) Z=\pm \mu Z_\eta,\ \ -\infty<\eta<\infty,
\end{equation}
has a solution $Z$.

We make a few more transformation for \eqref{210}. If we take
$
f: Z(\eta)=f(3\eta/(2L)),
$
it leads us  into the equivalent eigenvalue problem
\begin{equation}
\label{220}
-f''(x)- 6  \sech^2(x) f(x)=\pm \mu_1 f'(x),\ \ x\in \rone,
\end{equation}
where $\mu_1=\f{54 \mu}{L}$.

Next, we change variables to reduce \eqref{220} to a regular static Schr\"odinger equation, i.e. one where is no first derivative terms. To this end, fix the plus sign\footnote{the other case of $-\mu_1$ can be treated similarly and note that our statement is about instability, so we only need to show the solvability of \eqref{220} for $+\mu_1$}
 in \eqref{220}. Take $q: f(x)=e^{-\f{\mu_1}{2} x} q(x)$.  Plugging this into \eqref{220} results in the new equation in terms of $q$
 \begin{equation}
 \label{300}
 -q''(x) - 6  \sech^2(x) q(x) = -\f{\mu_1^2}{4} q(x).
 \end{equation}
In this form, we can solve the problem   explicitly. Indeed,
$$
\cl=-\p_{x}^2 - 6 \sech^2(x)=L_+ - 1
$$
where $L_+$ is the Schr\"odinger operator arising in the linearization around the soliton $sech(x)$ in the cubic NLS.  As such, we know that for $q_1:=-\sech'(x)=\sech(x)\tanh(x)$, $L_+[q_1]=0$ or in terms of $\cl$, for $\mu_1=2$, we have
\begin{equation}
\label{n:110}
-q_1''- 6  \sech^2(x) q_1 = -q_1.
\end{equation}

From this last formula  for $q_1$, it is clear that the ($2L$ periodization of the) corresponding function $z$ cannot be even continuous at $\pm L$. Indeed,
$$
\lim_{\xi\to -L+} z(\xi)= \lim_{\eta\to -\infty} Z(\eta)=\lim_{x\to -\infty} f(x)=
\lim_{x\to -\infty}  e^{-x} \sech(x)\tanh(x)=-2
$$
 On the other hand
 $$
\lim_{\xi\to L-} z(\xi)=
\lim_{x\to \infty} e^{-x} f(x)= 0.
$$
Hence, we establish that the $2L$ periodization of the function $z$ constructed here is not even continuous at $\pm L$. On the other hand, $z\in C^\infty(-L,L)$ and it satisfies the eigenvalue equation  \eqref{70} for $\xi \in (-L,L)$ by tracing back the changes of variables.

\end{proof}

\end{document}